\documentclass{amsart}
\usepackage{amsfonts}
\usepackage{amssymb}
\usepackage{amssymb}
\usepackage{enumerate}
\usepackage{latexsym}
\usepackage{nomencl,mathrsfs}
\usepackage{color}
\usepackage{graphicx}
\usepackage[all]{xy}
\usepackage{tikz, soul}
\usepackage{subsupscripts}
\usepackage{hyperref}
\usepackage{ifthen}

\input xy
\xyoption{all}

\newboolean{workmode}
\setboolean{workmode}{false}

\newcommand{\ifwork}[1]{\ifthenelse{\boolean{workmode}}{#1}{}}
\newcommand{\comment}[1]{}
\newcommand{\mute}[1]{}
\newcommand{\printname}[1]{}

\ifwork{
\renewcommand{\comment}[1]{{\marginpar{*}\ \scriptsize{#1}\ }}
\renewcommand{\printname}[1]
    {\smash{\makebox[0pt]{\hspace{-1.0in}\raisebox{8pt}{\tiny #1}}}}
}

\newcommand{\labell}[1] {\label{#1} \printname{#1}}

\definecolor{laura}{rgb}{.4, 0, .6}
\definecolor{jaw}{rgb}{0,.5,0}

\newtheorem{theorem}{Theorem}[section]
\newtheorem{lemma}[theorem]{Lemma}
\newtheorem{proposition}[theorem]{Proposition}

\newtheorem{corollary}[theorem]{Corollary}

\theoremstyle{definition}
\newtheorem{definition}[theorem]{Definition}
\newtheorem{remark}[theorem]{Remark}
\newtheorem{example}[theorem]{Example}

\numberwithin{equation}{section}

\newcommand{\bi}{\begin{itemize}}
\newcommand{\ei}{\end{itemize}}
\newcommand{\be}{\begin{enumerate}}
\newcommand{\ee}{\end{enumerate}}

\newcommand{\G}{\mathcal{G}}

\newcommand{\calB}{\mathcal{B}}

\newcommand{\N}{\mathbb{N}}
\newcommand{\Z}{\mathbb{Z}}
\newcommand{\Hom}{\text{Hom}}

\newcommand{\Quad}{\quad\quad\quad\quad\quad\quad\quad\quad\quad}

\newcommand{\id}{\operatorname{id}}
\newcommand{\g}{{\mathcal{G}}}

\newcommand{\gbb}{{\mathcal{G}_b^b}}
\newcommand{\gub}{{\mathcal{G}^b}}
\newcommand{\glb}{{\mathcal{G}_b}}
\newcommand{\Map}{\operatorname{Map}}
\newcommand{\pr}{\operatorname{pr}}
\newcommand{\Top}{{\operatorname{\textbf{Top}}}}
\newcommand{\ZZ}{{\mathbb{Z}}}

\newcommand{\toto}{{~\rightrightarrows~}} %groupoid double arrow

\newcommand{\ftimes}[2]{{\lrsubscripts{\times}{#1}{#2}}} %fiber product

\theoremstyle{remark}
\theoremstyle{definition}

\theoremstyle{definition}
\numberwithin{equation}{section}

\usepackage[margin=1in]{geometry}
\AtEndDocument{\bigskip{\small%
		\textsc{Carla Farsi: Department of Mathematics, University of Colorado at Boulder, Boulder, CO, USA.} \par
		\textit{E-mail address}: \texttt{carla.farsi@colorado.edu} \par
		\addvspace{\medskipamount}
		\textsc{Laura Scull: Department of Mathematics, Fort Lewis College, Durango, CO, USA.}\par
		\textit{E-mail address}: \texttt{scull\_l@fortlewis.edu}\par
		\addvspace{\medskipamount}
		\textsc{Jordan Watts: Department of Mathematics, Central Michigan University, Mt. Pleasant, MI, USA.} \par
		\textit{E-mail address}: \texttt{jordan.watts@cmich.edu}
}}

\title{Classifying Spaces and Bredon (Co)homology for Transitive Groupoids}
\author{Carla Farsi,   Laura Scull and Jordan Watts}

\date{\today}

\begin{document}

\begin{abstract} 
	We define the orbit category for transitive topological groupoids and their equivariant CW-complexes. By using these constructions we define equivariant Bredon homology and cohomology for actions of transitive topological groupoids.  We show how these theories can be obtained by looking at the action of a single isotropy group on a fiber of the anchor map, extending equivariant results for compact group actions.  We also show how this extension from a single isotropy group to the entire groupoid action can be applied to the structure of principal bundles and classifying spaces.
\end{abstract}

\maketitle

%\tableofcontents
%%%%%%%%%%%%%%%%%%%%%%%%%%%%%%%%%%%%%%%%%%%%%%%%%%%%%%%

%%%%%%%%%%%%%%%%%%
\section{Introduction}\label{s:intro}
%%%%%%%%%%%%%%%%%%

A groupoid, defined to be a small category where all arrows are invertible, is a natural generalization of a group.  Topological groupoids and Lie groupoids arise naturally in many subjects such as bundle theory, $C^*$ algebras, and mathematical physics.     One of the simplest examples is the pair groupoid of a topological space $X$, whose objects are points of $X$ and whose arrows are given by pairs $(x_1, x_2) \in X\times X$ with composition defined by $(x_1, x_2)(x_2, x_3) = (x_1, x_3)$.  

In this paper, we focus on the case of transitive topological groupoids where all objects are connected to each other by arrows. The pair groupoid above is an example of this.    These groupoids have been studied in various contexts.  An important class of examples of topological groupoids is that of gauge groupoids associated to principal bundles, and these groupoids  are also  transitive \cite{crainic}.   Work has been done on studying the structure of transitive groupoids and classifying them;  see for example \cite{razny} who studied Hilbert's fifth problem for transitive groupoids, and Mackenzie and Androulidakis  \cite{Mackenzie-class, Mackenzie-book, Androulidakis-connection, Androulidakis-classification} who used them in classifying principal bundles.    The related case of transitive Lie algebroids  has also been extensively studied, see \cite{chen-liu, fournel, misc}. 

When looking at actions of compact groups,  equivariant homotopy theory has developed a wide range of tools.  
Bredon homology and cohomology theories, first defined by Bredon \cite{Bredon},  have proved to be one of the most versatile, allowing fixed set  information to be incorporated.    This has led to results in equivariant obstruction theory, orientation theory, and covering spaces among others \cite{costenoble, costenoble2}.  

In this paper, we generalize Bredon homology and cohomology to the actions of transitive topological groupoids (see Definitions \ref{d:bredon-cohom}, \ref{d:bredon-hom}).    
A pivotal step in our construction is the generalization in Section \ref{s:trans gpd}  to transitive groupoids of many of the  structural formulas that apply to compact group actions \cite{May}.
These formulas allow us to define the orbit category for a transitive topological groupoid (see Definition \ref{d:orbit}), and consequently Bredon (co)homologies.

We develop results (see Propositions~\ref{p:formY} and \ref{p:rest}) showing how much of the equivariant theory associated to an isotropy subgroup extends naturally to the action of the entire groupoid.   We use this principle to define $\G$-CW-complexes which are used as the basis for our (co)homology theories.   We also show how this extension principle can be applied to the structure of classifying spaces.  

Every transitive groupoid with open source and target maps is Morita equivalent to the group defined by any of its (isomorphic) isotropy groups \cite{Williams}.  Therefore, the results of this paper can be seen as a partial Morita invariance result for the Bredon (co)homology theories.  Morita invariance results have shown up in other contexts, such as  in Pronk-Scull \cite{scull-pronk-translation}  who studied Morita invariance phenomena of equivariant cohomology theories, including Bredon cohomology and K-theory, in the context of orbifolds, and Williams \cite{Williams} who shows how to transfer Haar systems across Morita equivalence.  It is an interesting avenue for future research to see how generally this holds.  

$\G$-CW-structures for Lie groupoid actions have also been defined by Cantarero \cite{cantarero} while studying equivariant $K$-theory.  His groupoids are Lie but not transitive; in order to work in this generality, he makes an extension assumption that he calls `Bredon compatibility', and his CW-cells  are weakly equivalent to the action of a group on an entire finite CW-structure. Our CW-structures apply to transitive topological groupoids that are not Lie, and are much closer in spirit to the group case, where each cell carries constant orbit type.  This allows us to define  Bredon  theories which recover  information about the fixed sets of the space by  judicious choice of $\g$-coefficient system.  See Definitions \ref{d:coeff system} through \ref{d:bredon-hom}, and  Examples \ref{X:quotient} through \ref{X:bred}.

We also  derive some applications of Bredon homology to Smith theory when $\G$  acts on the space $X$ with anchor maps $a_X\colon X \to \G^0$ and  $a_{\G}:=(a_X)|_{X^\G}\colon X^\G \to \G^0$  fibrations. The results of 
  Smith theory for groups  extend to groupoid actions using the fact that 
  Euler characteristics with coefficients in a field are multiplicative on fibrations; see Theorem  \ref{them:A-Smith-theory} and Corollary \ref{cor:Smith-Theory}.
  Our definition of Bredon cohomology  also allows the development of  an obstruction theory for spaces with actions of transitive groupoids, also extending the group case;  see Remark \ref{th:obstruction}.

  The paper is organized as follows.  Section \ref{s:background} gives basic definitions, sets notation, and gives background material.  In Section \ref{s:trans gpd} we turn specifically to the study of actions of transitive groupoids, and prove the foundational results that set up the extension from group actions to transitive groupoid actions.  Section \ref{s:bredon} defines $\G$-CW-structures and Bredon (co)homology for actions of a transitive groupoid $\G$; Subsection~\ref{ss:examples} contains some examples and Subsection~\ref{ss:apps} contains an application to Smith Theory.  Finally, Section \ref{s:princ bdle} gives extension results about principal bundles.

{\bf{Acknowledgements.}} C.F.\ was partially supported by the
Simons Foundation  grant \#523991.  The authors thank Andrew Putman for helpful correspondence, and the anonymous referee for many excellent suggestions.

%%%%%%%%%%%%%%%%%%%%%%%%%%%%%%%%%%%%%
\section{Background: Groupoids and Groupoid Actions}\labell{s:background}
%%%%%%%%%%%%%%%%%%%%%%%%%%%%%%%%%%%%%

For the purposes of this paper, we assume that we are working in the category $\Top$ of Hausdorff compactly generated spaces and continuous maps (recall that a Hausdorff space $X$ is compactly generated if any subspace $A$ is closed in $X$ if and only if $A \cap K$  is closed in $K$ for all compact subspaces $K \subseteq X$).  This could be weakened to compactly generated weakly Hausdorff spaces, as is often done in CW-theory, but this would lengthen the proofs.
 
Most of the material in this section, including the notation, comes from \cite{renault}.  Recall the fibered product in the topological category: given spaces $X$, $Y$, and $Z$ with maps $f\colon X\to Z$ and $g\colon Y\to Z$, we have
$$X\ftimes{f}{g} Y:=\{(x,y)\in X\times Y\mid f(x)=g(y)\}.$$
We sometimes will denote this by $X\times_Z Y$ when the maps $f$ and $g$ are understood.
 
Let $\G$ be a groupoid.  Consider $\G$ defined by a space of objects $\g^0$, a space of arrows $\g^1$, and  structure maps defining the groupoid structure.  Explicitly, we have
\begin{itemize}
\item $i\colon  \g^0 \to \g^1$ sending an object to its identity arrow,
\item the source and target maps $s, t\colon \g^1 \to \g^0$,
\item the inverse map $\g^1 \to \g^1$ sending $f$ to $f^{-1}$, and
\item  the composition map $m\colon  \g^1\!\ftimes{s}{t} \g^1 \to \g^1$, where we denote composition by  $fg:=f\circ g$.
\end{itemize}
These structure maps are required to satisfy the standard groupoid axioms:
\begin{itemize}
\item if $f\colon  x \to y$, then $fi(x) = f$ and $i(y)f = f$,
\item $(hg)f = h(gf)$ for all composable $f$, $g$, and $h$,
\item $ff^{-1} = i(y)$, and
\item $f^{-1}f = i(x)$.
\end{itemize}

\begin{definition}[Topological Groupoid]\labell{d:top gpd}
A \textbf{topological groupoid} is a groupoid $\G$ in which $\G^0$ and $\G^1$ are Hausdorff, the structure maps are continuous, and $s$ (and hence $t$) is open. 
\end{definition}

\begin{definition}[Stabilizer]\labell{d:stabiliser}
Fix $b\in\G^0$.  The \textbf{stabilizer group} of $b$ is the group
$$\gbb := \{ g \in \g^1 \mid s(g) = t(g)  =b\}.$$
\end{definition}

\begin{definition}[Source Fiber]\labell{d:source fiber}
Fix $b\in\G^0$.  The \textbf{source fiber} of $b$ is the space 
$$\glb := \{ g \in \g^1 \mid s(g)  =b\}.$$
\end{definition}

\begin{definition}[Target Fiber]\labell{d:target fiber}
Fix $b\in\G^0$.  The \textbf{target fiber} of $b$ is the space 
$$\gub := \{ g \in \g^1 \mid t(g)  =b\}.$$
\end{definition}

We will additionally require one of the following two conditions:  either $s$ (and hence $t$) is proper, or $t|_{\glb}\colon\glb\to\g^0$ admits local sections.   We will specify which topological condition is in place for each result.  

\begin{remark}  Without some  topological conditions, we can get strange behavior; for example, see \cite{Buneci,Williams}.  The conditions of properness and local sections are closely related; in the Lie groupoid setting, one always has local sections of the source and target maps (as these are submersions),  and properness of these maps is often assumed to obtain nice properties; see, for instance, \cite{crainic-struchiner}.

The condition on $\g$ that $t|_{\glb}\colon\glb\to\g^0$ admits local sections is rather mild; it is equivalent to requiring $s|_{\gub}\colon\gub\to\G^0$ to be a principal $\gbb$-bundle.  This is guaranteed in the case, for example, that $\gbb$ is a topological manifold, $\g^1$ is first countable, and $s|_{\gub}$ is open.  This follows from the solution to Hilbert's fifth problem, and the relationship between topological groupoids and so-called Cartan principal bundles.  See, for example, \cite{razny,razny2}.  
\end{remark}

Given a groupoid $\G$, we can talk about actions of $\G$ on spaces, see for example \cite{Lerman, renault}. 

\begin{definition}[Right Action]\labell{d:right action}
A \textbf{right $\G$-space} is a space $X$ equipped with an \textbf{anchor map} $a_X\colon  X \to \g^0$ and an \textbf{action} $X\ftimes{a_X}{t} \G^1 \to X\colon (x,g)\mapsto xg$ satisfying the following conditions:
\begin{enumerate}
\item $a_X(xg) = s(g)$ for all $(x,g)\in X\ftimes{a_X}{t} \G^1$,
\item $x\, (i(a_X(x))= x$ for all $x\in X$, and
\item $(xg)g' = x(gg')$ for all $(x,g)\in X\ftimes{a_X}{t} \G^1$ and arrows $g'$ right-composable with $g$.
\end{enumerate}
We say that a right $\G$-space is \textbf{proper} if the anchor map $a_X$ is a proper map.
\end{definition} 

\begin{example}\labell{x:right action}
\noindent
\begin{enumerate}
			\item \labell{i:right action 1}
		The object space $\g^0$ is a right $\g$-space with anchor map $a_{\G^0} = \id_{\g^0}$, the identity map on $\g^0$, and action given by $yg = x$ for $g\colon x \to y$. 
		\item \labell{i:right action 2} Fix $b\in\g^0$.  The target fiber $\gub$ is a right $\g$-space with anchor map equal to the source map $s$, and action given by composition on the right.
		\end{enumerate}
\end{example}

\begin{definition}[Orbits and Orbit Space]\labell{d:orbitspace2}
 Let $X$ be a right $\G$-space.  Define an equivalence relation $\sim$ on $X$ by: $x\sim y\in X$ if there exists $g\in\g^1$ with $xg=y$. Denote by $[x]$ the equivalence class of $x$ with respect to $\sim$; we will also call $[x]$ the \textbf{orbit} of $x$. The \textbf{orbit space} $X/\g$ is the quotient space $X/\!\sim$.
 \end{definition}

\begin{definition} (Fixed Point Set \cite{deaconu})
\label{x:fixed point set}
Let $X$ be a right  $\G$-space  with anchor map $a_X: X \to \G^0$. The {\bf fixed set} $X^\G$ is defined by
\[
X^\G := \left\{ x \in X~\Big|~ gx=x, \ \forall  g \in \G_{a_X(x)}^{a_X(x)} \right\}.
\] 
\end{definition}

We also have notions of $\g$-equivariant maps, products, and push-outs.

\begin{definition}[Equivariant Map]\labell{d:equivariant}
A map $\psi\colon X\to Y$ between $\g$-spaces is a \textbf{$\G$-equivariant map} (or a \textbf{$\g$-map}) if it commutes with the anchor maps (\emph{i.e.}\ $a_X = a_Y\circ\psi$) and commutes with the actions (\emph{i.e.}\ $\psi(xg) = \psi(x)g$).  If additionally $\psi$ has a continuous inverse $\psi^{-1}$, then $\psi^{-1}$ is necessarily also a $\G$-map, and we say that $\psi$ is a \textbf{$\g$-homeomorphism}, and that $X$ and $Y$ are \textbf{$\g$-homeomorphic}.
\end{definition}
 
\begin{remark}\labell{r:left action}
Fix $b\in\g^0$. The target fiber $\gub$ admits a {\bf left action} of the group $\gbb$ via left composition, making it a $\gbb$-equivariant space in the classical group sense.  This action commutes with the right action in Item~\ref{i:right action 2} of Example~\ref{x:right action}; hence the quotient $\gub/\gbb$ inherits a well-defined action of $\g$ given by $[f]g:=[fg]$ for composable arrows $f$ and $g$, with injective anchor map $\widetilde{s}\colon\gub/\gbb\to\g^0$ defined by $\widetilde{s}([f])=s(f)$.  Moreover, with respect to the right $\g$-action on $\g^0$ of Item~\ref{i:right action 1} of Example~\ref{x:right action}, $\widetilde{s}$ is $\g$-equivariant.
\end{remark}

\begin{definition}[Product]\labell{d:prod}
Given two $\G$-spaces $X$ and $Y$, define the \textbf{product} $\G$-space as 
$$X \times_{\g^0} Y = \{ (x, y) \in X \times Y \mid a_X(x) = a_Y(y) \},$$
equipped with anchor map $a(x, y) := a_X(x) = a_Y(y)$ and action defined by $(x, y)g := (xg, yg)$ for $g$ satisfying $t(g) = a_X(x) = a_Y(y)$.
\end{definition}

\begin{definition}[Pushout]\labell{d:po} Given $\g$-maps $\varphi\colon  X \to Y$ and $\psi\colon  X \to Z$, define the \textbf{pushout} of $\varphi$ and $\psi$ to be the space $\left(Y \amalg Z\right)/{\sim}$, where $\sim$ is the equivalence relation generated by the relation $y \sim z$ if there exists $x \in X$ such that $\varphi(x) = y $ and $\psi(x) = z$.  The anchor map is given by $a([y]) := a_Y(y)$ and $a([z]) := a_Z(z)$ for $y\in Y$ and $z\in Z$.  This is well-defined, since if $y \sim z$ then there exists $x$ such that $\varphi(x) = y $ and $\psi(x) =  z$, and so $a_Y(y) = a_Z(z) = a_X(x)$.  The $\G$-action is given by using the action on $Y$ and $Z$: if $y \in Y$ and $t(g) = a(y)$ then we have $$([y],g) \in \left(\left(Y \amalg Z\right)/\!\sim\right)\ftimes{a}{t}\G^1$$ and we define $[y]g = [yg]$; we similarly define $[z]g:=[zg]$ for $z\in Z$.  This is also well-defined: if $y \sim z$ then $$yg = \varphi(x)g = \varphi(xg) \sim \psi(xg) = \psi(x) g = z g.$$  
\end{definition}

%%%%%%%%%%%%%%%%%%%%%%%%%%%%%
\section{Actions of Transitive Groupoids}\label{s:trans gpd}
%%%%%%%%%%%%%%%%%%%%%%%%%%%%%

We now turn to the study of actions of transitive groupoids, the main objects of this paper. In this section we generalize many of the structural formulas that apply to compact group actions \cite{May} to actions of transitive groupoids.  We begin by recalling their definition. 

\begin{definition}[Transitive Groupoid]\labell{d:trans gpd}
A groupoid $\G$ is \textbf{transitive} if $\G^0 /\G^1 $ is a single point.
\end{definition}

\begin{remark}\labell{r:trans gpd}
Any right action of a transitive groupoid has a surjective anchor map.
\end{remark}

\begin{example} Examples of transitive groupoids are pair groupoids, fundamental groupoids, and the gauge groupoid  (also called the Atiyah groupoid) $\G:= \Big( P \times P\Big)/G$ associated to a principal $G$-bundle $P$ (where $G$ is a group).
\end{example}

\begin{lemma}\labell{l:L1}
Let $\G$ be a transitive groupoid, and fix $b\in\G^0$.   The quotient $\gub/\gbb$ is $\g$-homeomorphic to $\g^0$, where the target fiber $\gub$  admits the canonical standard (left) action of the group $\gbb$.
\end{lemma}
 
\begin{proof}
By Remark \ref{r:left action}, $\gub/\gbb$ is a right $\g$-space with action $[f]g=[fg]$ and injective anchor map $\widetilde{s}([f]):=s(f)$ which is $\g$-equivariant with respect to the right $\g$-action on $\g^0$.  Moreover, since $\g$ is transitive, $\widetilde{s}$ is surjective.  Since $s$ is open, $\widetilde{s}$ is also open, and hence a homeomorphism.
\end{proof}

\begin{remark}\labell{r:L1}
Note that we use the openness of the source map in the proof above.  We could replace this condition with the requirement that $t$ be proper, in which case $\gub$ and hence $\gub/\gbb$ are compact.  Since $\g^0$ is Hausdorff, openness of $\widetilde{s}$ follows.
\end{remark}

\begin{remark}[Anchor Fibers]\labell{r:anchor fibers}
Let $Y$ be a right $\g$-space with anchor map $a$, and fix $b \in a(Y)$.  Denote the fiber $a^{-1}(b)$ by $Y_b$.   Then the $\g$-action on $Y$ restricts to a right $\gbb$-action on $Y_b$.

Conversely, given any right $\gbb$-space $X$, we can define a right $\g$-space:  recall from Remark \ref{r:left action} that $\gub$ has a {\em left} $\gbb$-action.  The \textbf{anti-diagonal action} of $\gbb$ on $X\times\gub$ is given by $g(x,f):=(xg^{-1},gf).$  This action commutes with the right action of $\g$ on $X\times\gub$ given by $(x,f)f':=(x,ff')$ with anchor map $(x,f)\mapsto s(f)$.  Thus the orbit space $X\times_{\gbb}\gub$ of the anti-diagonal action is a right $\g$-space with anchor map $a([x,f])=s(f)$ and action $[x,f]f' = [x,ff']$.
\end{remark}

For a transitive groupoid $\g$, any right $\g$-space is essentially an orbit space $X\times_{\gbb}\gub$ as in Remark~\ref{r:anchor fibers}.

\begin{proposition}\labell{p:formY}
Let $\g$ be a transitive groupoid with proper target map, fix $b\in\g^0$, and let $Y$ be a proper right $\g$-space.  The quotient $Y_b \times_{\gbb} \gub$ is $\g$-homeomorphic to $Y$.  Moreover, the construction of $Y_b\times_\gbb\gub$ given $Y$ is natural in $Y$.
\end{proposition}

\begin{proof}
Let $a$ be the anchor map of $Y$. The action map $Y\!\ftimes{a}{t}\G^1\to Y$ restricts to a continuous map $(Y_b) \ftimes{a}{t} \G^1=Y_b\times\gub \to Y$.
This restriction is constant on orbits of the anti-diagonal action of $\gbb$ on $Y_b\times\gub$, and so descends to a continuous map $I\colon Y_b\times_{\gbb}\gub\to Y$.  This map is $\G$-equivariant and injective, and surjectivity follows from the transitivity of $\G$.  Since $a$ and $t$ are proper, $Y_b\times\gub$ and hence $Y_b\times_{\gbb}\gub$ are compact.  Since $Y$ is Hausdorff, $I$ is a $\G$-homeomorphism.

For naturality, let $F\colon X\to Y$ be a $\g$-map, and denote by $I_X$ (resp.\ $I_Y$) the $\g$-homeomorphism constructed above from $X_b\times_\gbb\gub$ (resp.\ $Y_b\times_\gbb\gub$) to $X$ (resp.\ $Y$).  The restriction to $X_b$ induces a $\gbb$-map $F_b\colon X_b\to Y_b$.  Since $F_b\times\id_{\gub}$  is $\gbb$-equivariant with respect to the anti-diagonal actions on $X_b\times\gub$ and $Y_b\times\gub$, and also $\g$-equivariant with respect to the right $\g$-actions $(x,g)g'=(x,gg')$ and $(y,g)g'=(y,gg')$ for $x\in X$ and $y\in Y$, it descends to a $\g$-map $\widetilde{F}\colon X_b\times_\gbb\gub\to Y_b\times_\gbb\gub$ such that $F\circ I_X=I_Y\circ\widetilde{F}$.  This proves naturality.
\end{proof}

\begin{remark}  \labell{i:formP}  This result holds when neither the target map $t$ of $\g$ nor $Y$ are necessarily proper, but instead $t|_{\glb}$ admits local sections.  Indeed, consider a $\g$-space $Y$ and define 
 $\Phi\colon Y\to Y_b\times_\gbb\gub$ by $\Phi(y)=[y_b,g]$ where $y_b\in Y_b$ and $g\in\gub$ such that $y_bg=y$.  Then $\Phi$ is well-defined.  To show that it is continuous, let $a\colon Y\to\g^0$ be the anchor map, fix $y\in Y$, let $V$ be an open neighborhood of $a(y)$ for which there exists a local section  $\tau\colon V\to t|_{\glb}^{-1}(V)$  of $t|_{\glb}$.  Then $\Phi(y)=[y\tau(a(y)),\tau(a(y))^{-1}]$, which is continuous on $a^{-1}(V)$, and independent of the local section $\tau$ chosen.  It follows that $\Phi$ is continuous.  As it is an inverse to the map $I\colon Y_b\times_{\gbb} \gub\to Y$ constructed in the proof of Proposition~\ref{p:formY}, $\Phi$ is a $\G$-homeomorphism.  
\end{remark}

So the structure of a proper right $\g$-space $Y$ is determined by a fiber $Y_b$ regarded as a right $\gbb$-space.   Similarly, a $\g$-map is determined by its restriction to an anchor fiber.  In fact, equipping spaces of maps with the compact-open topology, we have:

\begin{proposition}\labell{p:rest}
Let $\G$ be a transitive groupoid with proper target map, fix $b\in\g^0$, and let $X$ and $Y$ be proper right $\G$-spaces.  There is a homeomorphism from $\Map_{\G}(Y, X)$ to $\Map_{\gbb}(Y_b, X_b)$ that is natural in each variable.
\end{proposition}

\begin{proof}
We will make use of the following notation: if $Z_1,Z_2$ are spaces, $K\subseteq Z_1$ compact and $U \subseteq Z_2$ open, then $V(K,U)$ is the set of all continuous maps $Z_1\to Z_2$ sending $K$ to $U$; \emph{i.e.} a standard sub-basis element of the compact-open topology on $\Map_{\Top}(Z_1,Z_2)$.

The restriction map $r\colon\Map_{\Top}(Y,X)\to\Map_{\Top}(Y_b,X)$ is continuous, and it restricts to a continuous map $r\colon\Map_\G(Y,X)\to\Map_{\Top}(Y_b,X)$. The natural inclusion $\Map_{\Top}(Y_b,X_b)\to\Map_{\Top}(Y_b,X)$ is a topological embedding, and hence so is $\Map_{\gbb}(Y_b,X_b)\to\Map_{\Top}(Y_b,X)$. Since $r$ sends $\g$-maps $Y\to X$ to $\gbb$-maps $Y_b\to X_b\subseteq X$, it follows that $r$ induces a unique continuous map, also denoted $r$, sending $\varphi\in\Map_{\G}(Y,X)$ to its restriction $\varphi_b\in\Map_{\gbb}(Y_b,X_b)$.

Let $p\colon\Map_{\Top}(Y_b,X_b)\to\Map_{\Top}(Y_b\times \gub,X_b\times \gub)$ be the map sending $\psi$ to $(\psi,\id_\gub)$.  To show that $p$ is continuous, it is sufficient to show that $p^{-1}(V(K,U_1\times U_2))$ is open for $K\subseteq Y_b\times\gub$ compact, $U_1\subseteq X_b$ open, and $U_2\subseteq\gub$ open.  If this is empty, then we are done, so suppose otherwise for a fixed $K$, $U_1$, and $U_2$.  We will prove $p^{-1}(V(K,U_1\times U_2)) = V(\pr_1(K),U_1)$, where $\pr_1$ is the first projection map.  Note:
\begin{equation}\labell{e:rest}
(k,g)\in K \Rightarrow g\in U_2.
\end{equation}
If $(\psi,\id_\gub)\in V(K,U_1\times U_2)$ is in the image of $p$ and $k\in\pr_1(K)$, then \eqref{e:rest} implies there exists $g\in U_2$ such that $(k,g)\in K$.  Hence, $(\psi,\id_\gub)(k,g)\in U_1\times U_2$. So $\psi\in V(\pr_1(K),U_1)$. On the other hand, if $\psi\in V(\pr_1(K),U_1)$, then \eqref{e:rest} implies for any $(k,g)\in K$, we have $p(\psi)(k,g)=(\psi,\id_\gub)(k,g)\in U_1\times U_2$.  Thus $p(\psi)\in V(K,U_1\times U_2)$.  We conclude that $p^{-1}(V(K,U_1\times U_2))=V(\pr_1(K),U_1)$, and in turn that $p$ is continuous. 

Next, the restriction of $p$ to $\Map_{\gbb}(Y_b,X_b)$ is continuous.  Moreover, for $\psi\in\Map_{\gbb}(Y_b,X_b)$ the map $p(\psi)$ is a $\gbb$-map with respect to the anti-diagonal actions, and also a $\g$-map with respect to the right $\g$-actions.  By the universal property of subspaces, there exists a unique continuous map, also denoted by $p$,
$$p\colon\Map_\gbb(Y_b,X_b)\to\Map_\gbb(Y_b\times\gub,X_b\times\gub)\cap\Map_\g(Y_b\times\gub,X_b\times\gub)$$
sending $\psi$ to $(\psi,\id_{\gub})$.

Let $\pi_Y\colon Y_b\times\gub\to Y_b\times_{\gbb}\gub$ and $\pi_X\colon X_b\times\gub\to X_b\times_{\gbb}\gub$ be the quotient maps induced by the anti-diagonal actions.  Let $q\colon\Map_\gbb(Y_b\times\gub,X_b\times\gub)\to\Map_{\Top}(Y_b\times_\gbb\gub,X_b\times_\gbb\gub)$ be the map sending $\gbb$-maps to continuous maps between the quotients.  Let $K\subseteq Y_b\times_\gbb\gub$ be compact and $U\subseteq X_b\times_\gbb\gub$ be open.  Since $Y_b$ and $\gub$ are compact, $\pi_Y^{-1}(K)$ is compact.  It follows that $q^{-1}(V(K,U))$ is the set of all $\gbb$-maps in $V(\pi_Y^{-1}(K),\pi_X^{-1}(U))$; hence $q$ is continuous.  Since the anti-diagonal actions commute with the right $\g$-actions, by the universal property of subspaces, $q$ induces a continuous map (also called $q$)
$$q\colon\Map_\gbb(Y_b\times\gub,X_b\times\gub)\cap\Map_\g(Y_b\times\gub,X_b\times\gub)\to\Map_\g(Y_b\times_{\gbb}\gub,X_b\times_{\gbb}\gub).$$

Denote by $I_Y$ (resp.\ $I_X$) the $\g$-homeomorphism from $Y_b\times_{\gbb}\gub$ to $Y$ (resp.\ $X_b\times_{\gbb}\gub$ to $X$) given in Proposition~\ref{p:formY}, sending $[y,h]$ to $yh$.  Let $C_{Y,X}$ be the map sending a $\G$-map $\alpha\colon Y_b\times_{\gbb}\gub\to X_b\times_{\gbb}\gub$ to $I_X\circ\alpha\circ I_Y^{-1}\in\Map_{\g}(Y,X)$.  Since composition is a continuous operation, $C_{Y,X} $ is a continuous map.

$$\xymatrix{
\Map_{\gbb}(Y_b,X_b) \ar[rrr]^{p\Quad} & & & \Map_{\gbb}(Y_b\times\gub,X_b\times\gub)\cap\Map_{\G}(Y_b\times\gub,X_b\times\gub) \ar[d]^{q} \\
\Map_{\G}(Y,X) \ar[u]^{r} & & & \Map_{\G}(Y_b\times_{\gbb}\gub,X_b\times_{\gbb}\gub) \ar[lll]^{C_{Y,X}\quad}\\
}$$

We claim that $r$ and $C_{Y,X}\circ q\circ p$ are inverses of each other.  Indeed, for any $y\in Y$, we have $I_Y^{-1}(y)=[yh,h^{-1}]$ where $h\colon b\to a_Y(y)$; this is independent of the choice of $h$ (a similar equality holds for $I_X^{-1}$).  Thus, for $\varphi\in\Map_{\g}(Y,X)$ and $y\in Y$,
$$C_{Y,X}\circ q\circ p\circ r(\varphi)(y)=I_X([\varphi_b(yh),h^{-1}])=\varphi_b(yh)h^{-1},$$
where $\varphi_b$ is the restriction of $\varphi$ to $Y_b$.  In particular,
$$\varphi_b(yh)h^{-1}=\varphi(yh)h^{-1}=\varphi(y).$$
Similarly, for $\psi\in\Map_{\gbb}(Y_b,X_b)$, we have $r\circ C_{Y,X}\circ q\circ p(\psi)=\psi$.

Finally, let $Z$ and $W$ be right $\G$-spaces and $F\colon Y\to Z$ and $G\colon W\to X$ be $\g$-maps.  Denote by $r_{ZW}$ the restriction map $\Map_{\g}(Z,W)\to\Map_{\gbb}(Z_b,W_b)$ obtained above.  Then letting $F^*$ (resp.\ $G_*$) be the continuous operation of pre- (resp.\ post-) composition by $F$ (resp.\ $G$), and noting for $\varphi\in\Map_\g(Z,X)$ the equalities $r_{YX}(F^*\varphi)=r_{ZX}(\varphi)\circ r_{YZ}(F)=\varphi_b\circ F_b$, we have the two commutative diagrams below.  Naturality in both variables $X$ and $Y$ follows.

$$\xymatrix{
\Map_{\g}(Y,X) \ar[r]^{r_{YX}~~} & \Map_{\gbb}(Y_b,X_b) & & \Map_{\g}(Y,W) \ar[r]^{r_{YW}~~} \ar[d]_{G_*} & \Map_{\gbb}(Y_b,W_b) \ar[d]^{(G_b)_*} \\
\Map_{\g}(Z,X) \ar[r]_{r_{ZX}~~} \ar[u] ^{F^*} & \Map_{\gbb}(Z_b,X_b) \ar[u]_{F^*_b} & & \Map_{\g}(Y,X) \ar[r]_{r_{YX}~~} & \Map_{\gbb}(Y_b,X_b) \\
}$$
\end{proof}

\begin{remark}\label{r:rest}
We will later need a version of Proposition~\ref{p:rest} where we do not assume the properness of $t$ nor the anchor maps, but instead assume the existence of local sections of $t|_{\glb}$ and that $\gbb$ is compact.  This appears as Lemma~\ref{i:rest}.
\end{remark}

In a similar vein, we can use these results to show that the quotient spaces of a $\G$-space and its fiber agree. 

\begin{proposition}\labell{p:quots}
Let $\G$ be a transitive groupoid with proper target map, fix $b\in\g^0$, and let $Y$ be a proper right $\G$-space.  Then $Y/\G$ is homeomorphic to $Y_b/\gbb$.
\end{proposition}

\begin{proof}
Let $i\colon Y_b\to Y_b\times_\gbb\gub$ be the injection $y\mapsto [y,i(b)]$, where $i(b)$ is the identity arrow at $b$.  Since $\gbb$-orbits map into $\G$-orbits, $i$ descends to a continuous map $j\colon Y_b/\gbb\to \left(Y_b\times_\gbb\gub\right)\!/\G$.  

On the other hand, let $p\colon Y_b\times\gub\to Y_b$ be the projection map.  This is $\gbb$-equivariant, and so descends to a continuous map $q\colon Y_b\times_{\gbb}\gub\to Y_b/\gbb$.  Also, $q$ is $\G$-invariant with respect to the right $\G$-action on $Y_b\times_{\gbb}\gub$, so $q$ descends to a continuous map $r\colon \left(Y_b\times_\gbb\gub\right)\!/\G\to Y_b/\gbb$.

It is straightforward to check that $j$ and $r$ are inverses of one another.  By Proposition~\ref{p:formY}, $Y$ is $\G$-homeomorphic to $Y_b\times_\gbb\gub$, and so they have homeomorphic orbit spaces via the $\G$-actions.  
\end{proof}

\begin{example}\label{X:quots}
Let $(M,\omega)$ be a connected symplectic manifold equipped with a free right Hamiltonian action of a compact Lie group $G$, with proper $G$-equivariant momentum map $\Phi\colon M\to\mathfrak{g}^*$ with respect to the coadjoint action of $G$ on $\mathfrak{g}^*$ (which we take to be a right action).  Fix a regular value $b\in\mathfrak{g}^*$ of $\Phi$, let $\mathcal{O}$ be its coadjoint orbit, and set $Y=\Phi^{-1}(\mathcal{O})$.  Finally, let $\G$ be the action groupoid $\mathcal{O}\times G\toto\mathcal{O}$, which is transitive. 

$Y$ is a $\G$-space, where the $\G$-action is simply the $G$-action (recall that $Y$ is $G$-invariant since $\Phi$ is $G$-equivariant) and the anchor map is $\Phi$.  From Proposition~\ref{p:quots} we obtain a homeomorphism between two standard representations of the symplectic quotient with respect to $b$; namely, $\Phi^{-1}(b)/\gbb$ and $\Phi^{-1}(\mathcal{O})/G$.  This homeomorphism is also a consequence of a well-known fact in symplectic geometry (that the two quotients are symplectomorphic); see \cite[Theorem 6.4.1]{OR}.
\end{example}

We now look at proper $\g$-spaces which have trivial actions.  

\begin{proposition}\labell{triv} Let $\G$ be a transitive groupoid with proper target map, and let $Y$ be a proper right $\G$-space with anchor map $a$.  The following are equivalent:  

\begin{enumerate}
\item There exists $b\in\g^0$ such that $Y$ is $\g$-homeomorphic to  $Y_b \times \g^0$ with anchor map $(y,b')\mapsto b'$ and action $(y, b') g = (y, b'')$  for $g\colon  b'' \to b'$.
\item Given $y \in Y$,  for any  $g, g'  \in \g^1$ with $t(g) = t(g') = a(y)$ and $s(g)=s(g')$,  we have  $yg = yg'$.  
\item  For all $b \in a(Y)$, the action of  $ \gbb$ on $Y_b$ is trivial. 

\item   For some $b \in a(Y)$, the action of $\gbb$ on $Y_b  $ is trivial.  
\end{enumerate}
\end{proposition}

\begin{proof}
That $(1) \Rightarrow (2) \Rightarrow (3) \Rightarrow (4)$ is clear.  To show that $(4) \Rightarrow (1)$, note that if $\gbb$ acts trivially on $Y_b$, then $Y_b\times_{\gbb}{\gub}$ is $\g$-homeomorphic to $Y_b\times (\gub/\gbb)$.  The result follows from Proposition~\ref{p:formY} and Lemma~\ref{l:L1}.
\end{proof}

\begin{definition}[Trivial $\G$-Action]\labell{d:triv}
A proper right $\G$-space $Y$ has a \textbf{trivial} $\G$-action if the conditions of Proposition~\ref{triv} hold.
\end{definition}

\begin{lemma}\labell{l:fiberprod}
Let $\G$ be a transitive groupoid, fix $b\in\g^0$, and let $X$ and $Y$ be proper right $\G$-spaces.  If the $\G$-action on $Y$ is trivial, then the product  $X \times_{\G^0} Y$ of  Definition~\ref{d:prod} is $\G$-homeomorphic to $X\times Y_b$ with anchor map $(x,y)\mapsto a_X(x)$ and action $(x,y)g=(xg,y)$.
\end{lemma}

\begin{proof}
Assume that $Y$ is a trivial $\G$-space.  By Proposition~\ref{triv}, the map $Y\to Y_b\times\g^0$ sending $y$ to $(yg,b)$, where $g$ is any arrow $b\to a_Y(y)$, is a $\g$-homeomorphism.  Then $X\times_{\g^0} Y$ is $\G$-homeomorphic to $X\times_{\G^0} (Y_b\times\G^0)$.  Recalling that $(y,b')\mapsto b'$ is the anchor map on $Y_b\times\g^0$,
$$X\times_{\G^0}(Y_b\times \G^0) = \{(x,(y,b'))\mid a_X(x)=b'\}.$$   It follows that the map $X\times_{\G^0}(Y_b\times\G^0) \to X\times Y_b$ sending $(x,(y,b'))$ to $(x,y)$ is a $\G$-homeomorphism.  
\end{proof}

We can use the above lemma to define the notion of $\g$-homotopy for maps.  

\begin{definition}[$\G$-Homotopy]\labell{d:htpy} Let $\G$ be a transitive groupoid.  Denote by $\mathbb{I} = I \times \g^0$ the trivial $\g$-space where $I=[0,1]$, and let $f,g\colon X\to Y$ be $\G$-equivariant maps between $\G$-spaces $X$ and $Y$.  Then $f$ is \textbf{$\G$-homotopic} to $g$, denoted $f \sim g$, if there exists a $\g$-equivariant map $H\colon  X \times_{\g^0} \mathbb{I} \to Y$ such that  $H(x, 0 ) = f(x)$ and $H(x, 1) = g(x)$.  \end{definition}

The above results imply the following.

\begin{corollary}\labell{c:htpy}
Let $\G$ be a transitive groupoid with proper target map, fix $b\in\g^0$, and let $f,g\colon X\to Y$ be $\g$-homotopic maps between proper right $\G$-spaces. Then $f|_{X_b}$ and $g|_{X_b}$ are $\gbb$-homotopic maps.
\end{corollary}

\begin{proof}
By Lemma~\ref{l:fiberprod}, $X\times_{\g^0}\mathbb{I}$ can be identified with $X\times I$ with anchor map $a=a_X\circ\pr_1$. By Proposition~\ref{p:rest}, the homotopy $H$ between $f$ and $g$ restricts to a $\gbb$-equivariant homotopy $H\colon  X_b \times I \to Y_b$.
\end{proof}

%%%%%%%%%%%%%%%%%%%%%%%%%%%%%%%%%%%%%%%%%%%%%%%%%%%%%%%%%%%%%%%
\section{Bredon Homology and Cohomology}\label{s:bredon}
%%%%%%%%%%%%%%%%%%%%%%%%%%%%%%%%%%%%%%%%%%%%%%%%%%%%%%%%%%%%%%%

For Lie groups, the orbit category is defined to be the category of transitive $G$-spaces with $G$-equivariant maps between them \cite{Bredon, Lueck-book, mislin-valette}.   In this section we will extend this definition to transitive groupoids.

\begin{definition}[Transitive $\G$-Space]\labell{d:trans gspace}
Let $\G$ be a transitive groupoid, and let $X$ be a $\G$-space.  Then $X$ is a \textbf{transitive} $\G$-space if $X/\G$ is a single point.
\end{definition}

\begin{definition}[Orbit Category]\labell{d:orbit}
Let $\g$ be a transitive groupoid.  Define the \textbf{orbit category of $\G$}, denoted $O\G$, to have objects given by transitive $\G$-spaces, with $\G$-equivariant maps between them.
\end{definition}

We will show that this category is equivalent to the orbit category of the group $\gbb$ for a fixed $b\in\G^0$.  Observe  that $\gub$ is a transitive $\G$-space.  Moreover, if $H$ is a subgroup of $\gbb$, then the quotient $\gub/H$ via the induced left $H$-action is also a transitive $\G$-space.  We will show that every transitive $\G$-space is of this form.  

\begin{lemma}\labell{l:homog spaces}
Let $\G$ be a transitive groupoid with proper target map and $X$ a transitive right $\g$-space.  Fix $x\in X$ and let $b = a_X(x)$.  Define $H=\{g \in \g^1 \mid  xg = x \}$.  Then $X$ is a proper $\G$-space $\g$-homeomorphic to $\gub\!/H$, and $X_b$ is $\gbb$-homeomorphic to $\gbb/H$.
\end{lemma}

\begin{proof}
Define $\varphi\colon\gub\to X$ by $\varphi(f)=xf$.  Since $t(f)=b=a_x(x)$, this is well-defined.  Since the $\g$-action is continuous, $\varphi$ is continuous.  Since $\varphi(hf)=\varphi(f)$ for all $h\in H$, $\varphi$ descends to a continuous map $\widetilde{\varphi}\colon \gub/H\to X$.

It is straightforward to check that $\widetilde{\varphi}$ is injective, and surjectivity follows from the transitivity of the action.  It follows from the definitions that $\varphi$ is $\g$-equivariant, and this property descends to $\widetilde{\varphi}$ since the actions of $\gbb$ and $\g$ commute.

We have shown that $\widetilde{\varphi}$ is a continuous bijective $\g$-map.  Since $\gub$ (and hence $\gub/H$) is compact, and $X$ is Hausdorff, we conclude that $\widetilde{\varphi}$ is a $\G$-homeomorphism.  Properness follows from the fact that $t$ is proper, hence $\gub/H$ and $X$ are compact.
The last statement follows from the first since $(\gub/H)_b=\gbb/H$.
\end{proof}

The following corollary uses the topologies on the mapping spaces as in Proposition \ref{p:rest}.

\begin{corollary}\labell{c:2} 
For any proper $\G$-space $X$, the spaces $\Map_{\G}(\gub/H,X)$, $\Map_{\gbb}(\gbb/H,X_b)$, and $X_b^H$ are homeomorphic to each other.
\end{corollary}

\begin{proof}
The homeomorphism between the first two follows from Proposition~\ref{p:rest}.  The homeomorphism between the last two is a standard result from equivariant homotopy theory for actions of groups. 
\end{proof}

\begin{corollary}\labell{c:1}
The orbit category $O\g$ is equivalent to the orbit category $O\gbb$.
\end{corollary}

We can now use this to define $\g$-CW-complexes.   We follow the presentation of May \cite[I.3]{May} and \cite[Definition 1.1]{Lueck-survey} in defining $G$-CW-complexes for group actions.  

Let  $D^{n+1}$ be the closed $(n+1)$-disk with boundary the $n$-sphere $S^n$.  Fix $b\in \g^0$.  Given a closed subgroup $H$ of $\gbb$, an \textbf{$(n+1)$-cell} is the product of a $\g$-space $\gub/H$ with a trivial $\g$-space $D^{n+1}\times\g^0$; we identify this product with $\gub/H\times D^{n+1}$ (see Lemma~\ref{l:fiberprod}).  The \textbf{boundary} of this $(n+1)$-cell is similarly constructed, identified with $\gub/H\times S^n$, with $\g$-equivariant inclusion $j\colon\gub/H\times S^n\hookrightarrow \gub/H\times D^{n+1}$.

\begin{definition}[$\G$-CW-Complex]\labell{d:cw}  
Let $\g$ be a transitive groupoid with proper target map, and fix $b\in\g^0$.  A \textbf{$\G$-CW-complex} (with finite skeleta) is a right $\g$-space $X$ constructed inductively as an increasing union of skeleta 
\[ X^0 \subseteq X^1 \subseteq X^2 \subseteq \ldots \subseteq X^n\subseteq X^{n+1}\subseteq \ldots;\quad X=\bigcup_{n \in \N} X^n\]
  as follows:
\begin{itemize}
\item The set of \textbf{$0$-cells}, $X^0$, is a finite disjoint union of \textbf{canonical orbits} $$\coprod_{i\in I_0} \gub/H_i,$$ where $H_i$ is a  closed subgroup of $\gbb.$

\item 
We inductively construct the $(n+1)$-skeleton $X^{n+1}$ from the  $n$-skeleton $X^{n}$ of $X$ as follows: Given a finite collection of $(n+1)$-cells $\{\gub/H_{i}\times D^{n+1}\}_{i\in I_n}$ and a collection of $\g$-equivariant \textbf{attaching maps} $\{q_i^n\colon \gub/H_{i}\times S^n\to X^n\}_{i\in I_n}$, the \textbf{$(n+1)$-skeleton} $X^{n+1}$ of $X$ is the pushout of $\g$-spaces (see Definition~\ref{d:po}) as indicated in the diagram below, where we denoted by $j_i\colon\gub/H_i\times S^n\to \gub/H_i\times D^{n+1}$ the standard inclusion.  
$$
\xymatrix{
	\displaystyle{\coprod_{i\in I_{n}}} \gub/H_{i} \times S^n \ar[rr]^{\quad\quad\underset{i}{\coprod}\, q_i^n} \ar[d]^{\underset{i}{\coprod}\, j_i} & & X^n \ar[d] \\ 
	\displaystyle{\coprod_{i\in I_{n}}} \gub/H_{i} \times D^{n+1} \ar[rr]_{~~\quad\quad\underset{i}{\coprod}\, q_i^{n+1}} & & X^{n+1}\\}
$$
Let $i_n\colon X^n\to X^{n+1}$ be the inclusion map, and $\displaystyle{\coprod_i} Q_i^n $ the pushout map
$$\coprod_i Q_i^n\colon \coprod_{i\in I_{n}} \gub/H_{i} \times D^{n+1} \to X^{n+1}.$$
\item Finally, we define $X$ to be the colimit (in the category of spaces) of the collection of $n$-skeleta $X^n$ with their inclusion maps.
\end{itemize}
\end{definition}

\begin{remark}\labell{r:cw}
The fact that a $\g$-CW-complex $X$ is a right $\g$-space follows from the fact that the inclusion maps $i_n\colon X^n\to X^{n+1}$ are $\g$-maps, which is immediate from the definition of the pushout of $\g$-spaces.  Our requirement of finite skeleta implies that each $n$-skeleton is a proper $\G$-space.
\end{remark}

The convenience of this definition of $\g$-CW-complex is illustrated by the following version of Whitehead's Theorem for $\G$-spaces (see \cite{May}). 

\begin{theorem}\labell{t:whitehead}  Let $\g$ be a transitive groupoid with proper target map, and fix $b\in\g^0$.  A $\G$-map $f\colon  X \to Y$ between finite $\G$-CW-complexes is a $\G$-homotopy equivalence if and only if $f^H\colon  X_b^H \to Y_b^H$ is a weak equivalence for all closed subgroups $H$ of $\gbb$.   
\end{theorem}

\begin{proof}
The $\gbb$-version of Whitehead's Theorem states that a $\gbb$-map $f_b\colon X_b\to Y_b$ between $\gbb$-CW-complexes is a $\gbb$-homotopy equivalence if and only if the restrictions $f_b^H\colon X_b^H\to Y_b^H$ are weak equivalences for all closed subgroups $H$ of $\gbb$.  Thus, the  groupoid version of the theorem follows from Remark~\ref{r:cw} and Corollary~\ref{c:htpy}.
\end{proof}

We now connect $\g$-CW-complexes to $\gbb$-CW-complexes via the following proposition.

\begin{proposition}\labell{P:cw}
Let $\g$ be a transitive groupoid with proper target map, and fix $b\in\g^0$.  Then $X$ is a $\g$-CW-complex if and only if $X_b$ is a $\gbb$-CW-complex.
\end{proposition}

\begin{proof}
At each stage of the attachment, the map between proper $\G$-spaces $\coprod\gub /H \times S^n \to X^n$ is determined by a $\gbb$-equivariant map $\gbb/H \times S^n \to X^n_b$ by Proposition \ref{p:rest}, and the pushout of $\g$-spaces restricted to the fiber $X^{n+1}_b$ is given by the following $\gbb$-equivariant pushout.
$$
\xymatrix{
\coprod\gbb/H \times S^n \ar[r] \ar[d] & X^n_b \ar[d] \\ 
\coprod\gbb/H \times D^{n+1} \ar[r] & X^{n+1}_b\\}
$$
The result follows.
\end{proof}

In fact, the adjunctions used above show a stronger correspondence:  the $\G$-CW-cellular structure of a $\G$-CW-complex $X$ exactly corresponds to the $\gbb$-CW-cellular structure of the fiber $X_b$.   Furthermore, these adjunctions induce a bijection between the cellular structures of a $\G$-CW-complex $X$ and those of the $\gbb$-CW-complex $X_b$.\footnote{The authors thank the anonymous referee for pointing this out.}  This allows us to immediately generalize some of the constructions of \cite{May} to actions of transitive groupoids.  

Following \cite[Chapter 1, Section 4]{May}, note that any $\g$-space $X$ defines a contravariant functor $\Phi_X\colon  O\G \to \Top$, defined by $\Phi_X(Y) = \Map_{\G}(Y, X)$ for any transitive $\g$-space $Y$.

\begin{definition}[$\G$-Coefficient System]\labell{d:coeff system}
Let $\g$ be a transitive groupoid with proper target map, and fix $b\in\g^0$.  A contravariant (resp.\ covariant) \textbf{$\g$-coefficient system} is a contravariant (resp.\ covariant) functor from the orbit category $O\G$ to the category of abelian groups.
\end{definition}

For a $\g$-CW-complex $X$ with skeleta $X^n$, we define a functor 
$$\underline{C}_n(X) = \underline{H}_n( \Phi_{X^n}(\cdot),  \Phi_{X^{n-1}}(\cdot);  \Z),$$
whose value on $\gub/H$ is the relative homology $H_n(\Phi_{X^n}(\gub/H),\Phi_{X^{n-1}}(\gub/H);\Z)$.  This creates a contravariant $\g$-coefficient system.

By the transitivity of $\g$ and Corollary~\ref{c:2}, and the fact that $X^n$ is a proper $\G$-space for all $n$, 
$$\underline{C}_n(X)(\gub/H) = H_n( (X_b^n)^{H}, (X_b^{n-1})^{H}; \Z).$$
Note that as a functor, $\underline{C}_n(X)$ sends all of the maps in $O\G$ to maps between relative homology groups.  For example, the maps of $\gub/\{e\}$ to itself can be identified with $\gbb$, and so this induces an action of $\gbb$ on the corresponding relative homology group for the subgroup $\{e\}$.  

Recall the long exact sequence for relative homology groups using the triples $((X_b^n)^H,(X_b^{n-1})^H,(X_b^{n-2})^H)$, whose connecting homomorphisms induce natural transformations $d\colon\underline{C}_n(X)\to\underline{C}_{n-1}(X)$.
Using these, we now define the Bredon cohomology and homology of the $\G$-CW-complex $X$.   

\begin{definition}[Bredon Cohomology]\labell{d:bredon-cohom}
Let $\g$ be a transitive groupoid with proper target map, and fix $b\in\g^0$. Let $X$ be a $\g$-CW-complex and $M$ a contravariant $\G$-coefficient system. Define
\begin{equation*}\labell{eq-cell-compl-bred-cohomo}
C_{\G}^n(X, M) := \Hom _{O\G} \Big( \underline{C}_n(X), M \Big),\ \delta:=\Hom_{O\G} (d, \id_M ).
\end{equation*}
where $\Hom_{O\g} $ denotes natural transformations between the functor coefficient systems.  This induces a complex, whose homology is the \textbf{Bredon cohomology} of $X$ with coefficients in $M$.
\end{definition}
  
\begin{definition}[Bredon Homology]\labell{d:bredon-hom}
Let $\g$ be a transitive groupoid with proper target map, and fix $b\in\g^0$. Let $X$ be a $\g$-CW-complex and $N$ a covariant $\G$-coefficient system. Define
\begin{equation*}
\label{eq-cell-compl-bred-homo}
C^{\G}_n(X, N) := \Big( \underline{C}_n(X)  \Big) \otimes_{O\G}  N,\ \partial := d\otimes 1.
\end{equation*}
This induces a complex, whose homology is the \textbf{Bredon homology} of $X$ with coefficients in $N$.
\end{definition}

Thus we have defined Bredon cohomology and homology for $\g$-CW-complexes, analogous to the group case.  But by Corollary~\ref{c:1}, the indexing category $O\g$ is equivalent to $O\gbb$.  Therefore our theory for a $\G$-space $X$ will coincide with the $\gbb$-equivariant theory of the fiber $X_b$.    This, combined with Corollary~\ref{c:htpy} and the fact that the $\gbb$-theory is invariant under $\gbb$-homotopy \cite[I.4]{May}, also implies that these homology and cohomology theories will be invariant under $\g$-homotopy.

\begin{remark} \label{th:obstruction}  When $\g$ is proper Lie, our definition of Bredon cohomology also allows the development of  an obstruction theory for $\g$-spaces. For a $\g$-CW complex, the attaching maps are determined by $\gbb$-equivariant maps of fibers $\gbb/H \times S^n \to X_b$.  Given a map $f\colon  X^{n} \to Y$, we can compose the attaching maps with $f$ to create a cocycle $c_f \in C^{n+1}(X_b; \pi_{n}(Y_b))$, whose associated cohomology class will vanish if and only if the map $f$ can be extended to $X^{n+1}$ without changing it on $X^{n-1}$.  See \cite{May} for more details.
\end{remark}

%%%%%%%%%%%%%%%%%%%%%%%%%%%%%%%%%%%%%%%%%%%%%%%%%%%%%%%%%%%%%%%%%%%%%%%%%
\subsection{Examples}\labell{ss:examples}
%%%%%%%%%%%%%%%%%%%%%%%%%%%%%%%%%%%%%%%%%%%%%%%%%%%%%%%%%%%%%%%%%%%%%%%%%

We now  give examples of calculating Bredon cohomology for $\G$-spaces.

\begin{example}\labell{X:quotient}
Let $\g$ be a transitive groupoid with proper target map, and fix $b\in\g^0$.  Assume $\gbb$ is finite.  Let $X$ be a $\g$-CW-complex, and let $M = \underline{\mathbb{Z}}$ denote the coefficient system with values $M(\g^b/H) = \mathbb{Z}$ for all $H \subseteq \gbb$, with all structure maps given by the identity.   Consider the complex defined by  $\Hom_{O\G} \Big( \underline{C}_n(X), M \Big)$.   The diagram $\underline{C}_n(X)$ will have structure maps from $ \underline{C}_n(X)(\gbb/H)$  to itself generated by the action of $\gbb$, but since  the structure maps of $M$ are the identity, any functor from $\underline{C}_n(X)$ to $M$ will factor through the  quotient, and $\Hom_{O\G} \Big( \underline{C}_n(X), M \Big)$ will be isomorphic to $\Hom \Big(C_n(X/\G), \mathbb{Z}\Big)$.   Then the Bredon cohomology of $X$ with coefficients given by $M$ is isomorphic to the singular homology of $X/\G$ with coefficients in $\mathbb{Z}$.  
\end{example}

\begin{example}\labell{X:dbl}  We consider the transitive groupoid $\G$ given as an example in \cite{sieben}:  $\g^0 = \{ a, b\}$ and   $\G^1  = \{x,y, x^{-1}, y^{-1}, u =  i(a),v = i(b),s =  y^{-1}x,t = yx^{-1} \}$ with relations $s^2=u$ and $t^2=v$, where the structure is pictured as follows:  

$$ 
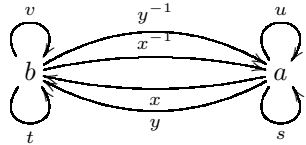
\xymatrix{ b \ar@(ul,ur)^{v}  \ar@(dl,dr)_{t}   \ar@/^/[rrr]^{x^{-1}} \ar@/^1.3pc/[rrr]^{y^{-1}} && & a \ar@/^/[lll]^x \ar@/^1.2pc/[lll]^y   \ar@(ul,ur)^{u}  \ar@(dl,dr)_{s}}  $$

This is a finite groupoid equipped with the discrete topology.  Recall that the $n$-cells for a $\g$-CW-complex are given by $\gub / H \times D^n$ for $H \leq \gbb$.   In our case,  $\gbb = \{ v, t\} \cong  \mathbb{Z}/2$, and $\gub = \{v, t, x, y\}$.  We will construct $X$ as follows: start with two $0$-cells   $X^0 = \gub/ \gbb \times D^{0} \coprod \gub/ \gbb \times D^{0}  $.  Each $0$-cell is two discrete points; we will denote the points of the first cell by $\blacksquare_a$ and $\blacksquare_b$, and the points of the second by $\blacktriangle_a$ and $\blacktriangle_b$, where the subscript denotes the image of the point under the anchor map.   We attach a single $1$-cell $\gub/\{e\} \times D^{1}$, which is  $4$ lines, via the $\g$-map that takes $\gub/\{e\} \times S^{0}$ to $X^0$ such that $\gub/\{e\} \times \{-1\} \to \{\blacksquare_a,\blacksquare_b\}$, and $\gub/\{e\} \times \{1\} \to \{\blacktriangle_a,\blacktriangle_b\}$.  The result is depicted as follows:  

$$ 
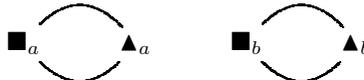
\xymatrix{ \blacksquare_a  \ar@/^1.2pc/@{-}[r]  \ar@/_1.2pc/@{-}[r] & \blacktriangle_a  & \blacksquare_b  \ar@/^1.2pc/@{-}[r]  \ar@/_1.2pc/@{-}[r] & \blacktriangle_b  }  $$
\vspace{0.125in}

Since $\g$ is discrete, this is just two copies of the fiber $X_b$, which is the $\ZZ/2$-space given by reflections of the circle $S^1$ through a fixed axis.

\end{example}

\begin{example}  \labell{X:hol} We construct a continuous version of the space described in Example \ref{X:dbl}.    Following \cite{crainic}, starting with the principal $\ZZ/2$-bundle given by the double cover $S^1\to S^1/(\ZZ/2)$, we construct a transitive groupoid $\g$ with $\g^0=S^1/(\ZZ/2)$ and $\g^1=(S^1 \times S^1)/(\mathbb{Z}/2)$, the orbit space of the diagonal $\ZZ/2$-action on $S^1\times S^1$ in which $[z,w]=[-z,-w]$.   The source and target maps are the second and first projection maps, respectively, and composition is defined by $[z_1,w_1][z_2,w_2] = [z_1,w_2]$.  Fixing $b \in S^1$, we have $\gbb \cong \mathbb{Z}/2$.  

We construct a $\g$-space $Y = Y_b \times_{\gbb} \gub$ by taking $Y_b$ to be the same fiber as in Example \ref{X:dbl}, the circle with action by reflections through a fixed axis.  The result is homeomorphic to the Klein bottle, with $\G$-CW-structure modeled on the $\mathbb{Z}/2$-structure of the fiber, having two $0$-cells created by $\gub/\gbb \times D^0$, both homeomorphic to the circle with trivial $\Z/2$-action, and one $1$-cell $\gub/\{e\} \times D^1$, homeomorphic to a cylinder equipped with an action of $\Z/2$ induced by rotation by $\pi$ around the circle $\gub/\{e\}$.

\end{example}

\begin{example}\labell{X:bred}  We  use the $\g$-CW-structure of Examples \ref{X:dbl} and \ref{X:hol} to calculate Bredon cohomology for a couple of coefficient systems.  Both of these spaces will have orbit categories equivalent to $\G^b/\{e\} \to \G^b/\gbb$ with a $\mathbb{Z}/2$ action on $\G^b/\{e\}$, and their  $\g$-CW-complexes are constructed with cells which correspond:   $\underline{C}_0(X) = (\mathbb{Z} \oplus \mathbb{Z} \xrightarrow{~\operatorname{id}~}  \mathbb{Z} \oplus \mathbb{Z})$, with identity map and trivial $\mathbb{Z}/2$-action,  and $\underline{C}_1(X) = (0 \hookrightarrow \mathbb{Z} \oplus \mathbb{Z})$ with a  non-trivial action of $\mathbb{Z}/2$ that swaps the copies of $\mathbb{Z}$.   

If we consider the coefficient system $\underline{A} = \mathbb{Z} \to 0$,  we calculate that $\Hom_n(\underline{C}_n(X), \underline{A}) $ is given by the chain complex $ \mathbb{Z} \oplus \mathbb{Z}  \rightarrow  0$.   In this case,  the Bredon cohomology recovers the cohomology of the fixed set $X_b^{\mathbb{Z}/2}$.    If we consider the cofficient system $\underline{B} =  \mathbb{Z} \xrightarrow{~\operatorname{id}~} \mathbb{Z}$, we find that $\Hom_n(\underline{C}_n(X), \underline{B}) $  is given by the chain complex $ \mathbb{Z} \oplus \mathbb{Z}  \twoheadrightarrow  \mathbb{Z} $  and we have recovered the cohomology of the quotient space $X/\G$, homeomorphic to $X_b/\gbb$.  
\end{example}

%% %% %%
\subsection{Bredon Homology and Smith Theory}\labell{ss:apps}
%% %% %%

We will now derive some applications of Bredon homology to  Smith theory. In this section, all spaces are assumed to be finite CW-complexes (resp.\ finite $\g$-CW-complexes, finite  $\gbb$-CW-complexes). The following result, which is an adaptation of \cite[Remark 3.4]{deaconu}, tells us when $X^\G \not= \emptyset$. Its proof is similar to the proof of \cite[Remark 3.4]{deaconu}, and therefore omitted.

\begin{proposition}\label{prop:fixed-point-set-transitive-gpd}
Let $\G$ be a transitive groupoid and $X$ a $\G$-space.  Then $X^\G \not= \emptyset$
if and only if  the anchor map $a_X\colon X \to \G^0$ has a $\G$-equivariant section $\sigma\colon \G^0 \to X$ (\emph{i.e.}\ $\sigma(gu)=g\sigma(u)$.) 
\end{proposition}

\begin{definition}[Multiplicative Euler Characteristic]\labell{d:multiplicative}
Let $\pi\colon X\to Y$ be a fibration with generic fiber $F$ and let $\chi$ denote the Euler characteristic with respect to cellular homology with coefficients in a group $H$. We say that $\pi$ has \textbf{multiplicative $\chi$-Euler characteristic} if $\chi(X)=\chi(F)\chi(Y).$
\end{definition}

\begin{example}\labell{x:multiplicative} 
Let $\pi\colon X\to Y$  be a locally trivial fibration with fiber $F$ over a path-connected space $Y$. Suppose for a prime number $p$ the induced action of $\pi_1(Y)$ on $H_*(F;\mathbb{Z}/p\mathbb{Z})$ is trivial.  Then $\pi $ has  multiplicative $\chi_p$-Euler characteristic, where $\chi_p$ is the Euler characteristic with respect to mod-$p$ homology; see \cite[Theorem 9.16]{davis-kirk}.  For instance, the Hopf fibration $\pi\colon S^3\to S^2$ satisfies $\chi_p(S^3)=\chi_p(S^1)\chi_p(S^2)$ for any prime $p$. 
\end{example}

If we assume that anchor maps $a_X\colon X \to \G^0$ and $a_{\G}:=(a_X)|_{X^\G}\colon X^\G \to \G^0$ are fibrations having {$\chi_p$}-multiplicative Euler characteristics, then the results of \cite{putman} readily extend to transitive groupoid actions. Note that the results of \cite{putman} extend to the category of $\gbb$-CW complexes by using  the main result in \cite{May-Smith} in place of \cite[Theorem B]{putman}.

\begin{theorem}
\label{them:A-Smith-theory} (\emph{c.f.}\ \cite[Theorem A]{putman}) Let $\g$ be a transitive groupoid with proper target map and $\g^0$ path-connected; fix $b\in\g^0$.  Assume that $\gbb$ is a finite $p$-group for some prime number $p$.  Let $X$ be a $\g$-CW-complex with anchor map $a_X$. Suppose that  $a_X$ and  $a_{\G}:=a_X|_{X^\G}$ are fibrations with multiplicative $\chi_p$-Euler characteristics.  We have:

\begin{enumerate}
\item The Euler characteristics $\chi_p(X)$ and $\chi_p(X^\g)$ are integers, and 
$$ \chi_p( X ) \equiv \chi_p\left( X^\G \right)~\mod p.$$

\item  If the fibers of $a_X$ are mod-$p$ acyclic, then the fibres of $a_\g$ are mod-$p$ acyclic and hence non-empty; in particular, $X^\g$ is non-empty.

\item  If the fibers of $a_X$ are mod-$p$ homology $n$-spheres, then either there exists a non-negative integer $m\leq n$ such that the fibres of $a_\g$ are homology $m$-spheres, or $X^\g$ is empty.
 \end{enumerate}
\end{theorem}

\begin{proof} By Proposition~\ref{P:cw}, $X_b$ is a finite $\gbb$-CW-complex, and hence so is $X_b^\gbb$.  By \cite[Theorem A]{putman}, 
\begin{equation}\labell{eq:euler-char}
\chi_p(X_b)\equiv\chi_p\left(X_b^\gbb\right)~\mod p.
\end{equation}
Since $a_X$ and $a_\g$ have multiplicative Euler characteristics, it follows from Equation~\eqref{eq:euler-char} that
$$\chi_p(X)= \chi_p(X_b)\chi_p(\G^0 )\ \equiv\  \chi_p\left(X_b^\gbb\right)\chi_p(\G^0) = \chi_p(X^\G)~\mod p.$$
This proves the first claim.

To prove the second claim, suppose the fibers of $a_X$ are mod-$p$ acyclic.  Then $X_b$ is mod-$p$ acyclic.  It follows from \cite[Theorem A]{putman} that $X_b^\gbb$ is mod-$p$ acyclic.  The result follows.  The third claim is similar, keeping in mind that the fibers of $a_\g$ are homeomorphic and the fibers of $a_X$ are homeomorphic.
\end{proof}

By putting together Proposition \ref{prop:fixed-point-set-transitive-gpd} and Theorem \ref{them:A-Smith-theory},  we obtain:

\begin{corollary}\labell{cor:Smith-Theory} Keeping the hypothesis of Theorem~\ref{them:A-Smith-theory}, if a fiber of $a_X$ is mod-$p$ acyclic, then $a_X$ admits a $\g$-equivariant section.
\end{corollary}

%% %% %%

%%%%%%%%%%%%%%%%%%%%%%%%%%%%%%%%%%%%%%%%%%%%%%%%%%%%%%%%%%%%%%%%%%%%%%%%%%
\section{Principal $\G$-bundles and Classifying Spaces}\label{s:princ bdle}
%%%%%%%%%%%%%%%%%%%%%%%%%%%%%%%%%%%%%%%%%%%%%%%%%%%%%%%%%%%%%%%%%%%%%%%%%%

In this section, we look at principal $\g$-bundles for transitive groupoids $\g$ with mild conditions.  We begin by recalling the definition (see \cite[p.\ 11]{Jelenc}).  

\begin{definition}[Principal $\G$-Bundle]\labell{d:princ bdle}
Let $\G$ be a groupoid and $M$ a Hausdorff paracompact space. A \textbf{principal $\G$-bundle} over $M$ is a paracompact space $P$ equipped with a surjective map $\pi\colon P \to M$ and a right $\G$-action with anchor map $\epsilon\colon P \to \G^0$ such that
\begin{enumerate}
	\item  the map $\pi$ has local sections,
	\item  for every $p \in P$ and $g \in \G^1$ with $\epsilon(p)=t(g)$, we have $\pi(pg)=\pi(p)$, and
	\item  the following map is a homeomorphism:
	$$\mu \colon P \ftimes{\epsilon}{t} \G^1 \to P \times_{M}\!P\colon (p,g) \mapsto (p,pg).$$
\end{enumerate}
\end{definition}

If $\G=H$ is a group, Definition \ref{d:princ bdle} reduces to the standard definition of a principal $H$-bundle in which the third condition is replaced by the condition that the group acts freely and transitively on the fibres; see \cite[Definition 2.1]{rossi} for the Lie group case, which also holds in the topological category.

Since $\pi$ has local sections and is therefore a quotient map, the space $M$ above is homeomorphic to the orbit space $P/\G$. We will not require the properness of $\epsilon$ for most of this section, as it turns out that the benefits of compactness are replaced by the benefits of Conditions~(1) and (3).  Recall that Proposition~\ref{p:formY} holds when local sections replace the properness condition (see Remark~\ref{i:formP}).  We now show that Proposition~\ref{p:rest} also has an analogous version for principal bundles.  We will need the following lemma.

\begin{lemma}\labell{l:pi proper}
Let $H$ be a compact topological group and let $\pi\colon P\to M$ be a principal $H$-bundle.  Then $\pi$ is proper.
\end{lemma}

\begin{proof}
Fix a compact set $K\subseteq M$.  Let $\{U_\alpha\}$ be an open cover of $K$ such that for each $\alpha$, $\pi^{-1}(U_\alpha)$ is $H$-homeomorphic to $U_\alpha\times H$.  Since $M$ is paracompact, there exists a refinement $\{V_\alpha\}$ of $\{U_\alpha\}$ such that $\overline{V_\alpha}\subseteq U_\alpha$ for each $\alpha$ (see \cite[Lemma 41.6]{munkres}).  Since $K$ is compact, there is a subcover $\{V_{\alpha_1},\dots, V_{\alpha_n}\}$.  Since $\pi^{-1}(K\cap\overline{V_{\alpha_i}})$ is homeomorphic to $(K\cap\overline{V_{\alpha_i}})\times H$, which is compact for each $i$, and finite unions of compact sets are compact, we conclude that $\pi^{-1}(K)=\bigcup_i\pi^{-1}(K\cap\overline{V_{\alpha_i}})$ is compact.
\end{proof}

\begin{lemma}\labell{i:rest}
Let $\G$ be a transitive groupoid with compact stabilizers, fix $b\in\g^0$, and suppose that $t|_{\glb}\colon\glb\to\g^0$ admits local sections.
 Let $\pi\colon P\to M$ and $\rho\colon Q\to N$ be principal $\g$-bundles with anchor maps $\epsilon_P$ and $\epsilon_Q$, respectively, such that $P$ is paracompact and $b\in\epsilon_P(P)\cap\epsilon_Q(Q)$.  There is a homeomorphism from $\Map_\g(P,Q)$ to $\Map_\gbb(P_b,Q_b)$ that is natural in each variable.

\end{lemma}

\begin{proof}
All arguments in the proof of Proposition~\ref{p:rest} carry through except the argument showing that the map $q$ is continuous.  In particular, we need to show that the quotient map $\pi_P\colon P_b\times\gub\to P_b\times_\gbb\gub$ is proper.  By Remark \ref{i:formP}, $P$ is homeomorphic to $P_b\times_{\gbb}\gub$, and so by Lemma~\ref{l:pi proper}, it is enough to show that $\pi_P$ is a principal $\gbb$-bundle, which we now show.

Fix $[p,g]\in P_b\times_\gbb\gub$, and let $\epsilon$ be the anchor map of the principal $\g$-bundle $P_b\times_\gbb\gub\to M$.  Let $V$ be an open neighborhood of $\epsilon([p,g])$ for which there exists a local section $\tau\colon V\to t|_{\glb}^{-1}(V)$  of $t|_{\glb}$.  Then the map $\sigma\colon\epsilon^{-1}(V)\to\pi_P^{-1}(\epsilon^{-1}(V))$ defined by
\[
 [p',g']\mapsto (p'g'\tau(\epsilon(p'g')),\tau(\epsilon(p'g'))^{-1})\]
is a well-defined continuous section of $\pi_P|_{\pi_P^{-1}(\epsilon^{-1}(V))}$.  Condition~(2) is straightforward to check, and Condition~(3) follows from the fact that $\pi\colon P\to M$ is a principal $\g$-bundle.
\end{proof}

\begin{proposition}\labell{p:restric-princ-bundle}  
Let $\G$ be a transitive groupoid, fix $b\in\g^0$, and suppose that $t|_{\glb}\colon\glb\to\g^0$ admits local sections. If $\pi\colon P\to M$ is a principal $\G$-bundle with anchor map $\epsilon$ such that $b\in\epsilon(P)$, then the restriction $\pi|_{P_b}\colon P_b\to M$ is a principal $\gbb$-bundle.  Conversely, if $\pi_b\colon P_b\to M$ is a principal $\gbb$-bundle, then we can extend $\pi_b$ to a principal $\G$-bundle $\pi'\colon P'\to M$.  Moreover, these operations are inverses of each other up to isomorphisms of principal bundles.
\end{proposition}

\begin{proof} 
Suppose that $\pi\colon P \to M$ is a principal $\g$-bundle.  Fix $p\in P$, an open neighborhood $U\subseteq M$ of $\pi(p)$ with a local section $\sigma\colon U\to \pi^{-1}(U)$, and an open neighborhood $V\subseteq\g^0$ of $\epsilon(p)$ with a  local section $\tau\colon V\to t|_{\glb}^{-1}(V)$.  Denote by $W$ the open set $\sigma^{-1}(\pi^{-1}(U)\cap\epsilon^{-1}(V))$.  Then $\upsilon\colon W\to P_b$ defined by $\upsilon(w)=\sigma(w)\tau(\epsilon(\sigma(w)))$ is a well-defined continuous section $W\to\pi|_{P_b}^{-1}(W)$.  Also, $\pi(pg)=\pi(p)$ for every $p\in P_b$ and $g\in\gbb$.  We now need only to verify Condition~(3) of Definition~\ref{d:princ bdle}.

Since $\pi\colon P\to M$ is a principal $\g$-bundle, the homeomorphism $\mu\colon P \ftimes{\epsilon}{t} \g^1 \to P \times_{M}P$ sending $(p, g)$ to $(p, pg)$ restricts to a homeomorphism $\mu_b$ from $(P_b) \ftimes{\epsilon}{t} \gbb=P_b\times\gbb$ to its image, which is contained in $P_b \times_M P_b$.  Fix $(p, q) \in P_b \times_M P_b$.  Since $\mu$ is surjective, there exists $(p, g)\in P\ftimes{\epsilon}{t}\G^1$ such that $(p, pg) = (p,q)$.  But then $g \in \gbb$, hence $(p,g)\in P_b\times\gbb$.  Therefore, $\mu_b$ is a homeomorphism from $P_b\times\gbb$ to $P_b\times_M P_b$.

Conversely, suppose that $\pi_b\colon P_b\to M$ is a principal $\gbb$-bundle.  Then $P_b\times\gub$ is a right $\g$-space with anchor map $s\circ\pr_2$ (where $\pr_i$ is the $i$th projection map) and action $(p,g)g'=(p,gg')$.  Since this action commutes with the anti-diagonal action of $\gbb$ on $P_b\times\gub$, the $\g$-action descends to a right $\g$-action on $P':=P_b\times_{\gbb}\gub$.

Since $\pi_b\circ\pr_1$ is constant on $\gbb$-orbits, it descends to a surjective continuous map $\pi'\colon P'\to M$ sending $[p,g]$ to $\pi_b(p)$.  Fixing $x\in M$, there is an open neighbourhood $U$ of $x$ and a section $\upsilon\colon U\to\pi_b^{-1}(U)$.  Define $\sigma(x):=[\upsilon(x),i(b)]$, where $i(b)$ is the identity arrow at $b$; this defines a section $U\to(\pi')^{-1}(U)$.  It follows that $\pi'$ admits local sections.  It is immediate that $\pi'$ satisfies Condition~(2) of Definition~\ref{d:princ bdle}.  We now need only to show Condition~(3).

Define $\mu\colon P'\ftimes{\epsilon'}{t}\g^1\to P'\times_M P'$ by $\mu([p,g],g')=([p,g],[p,gg'])$.  Then $\mu$ is continuous, and it follows from the freeness of the $\gbb$-action on $P_b$ that $\mu$ is injective.  Given $([p_1,g_1],[p_2,g_2])\in P'\times_M P'$, we have $\pi'([p_1,g_1])=\pi'([p_2,g_2])$.  It follows that there exists some $h\in\gbb$ such that $p_2=p_1h$.  But then $\mu([p_1,g_1],g_1^{-1}hg_2)=([p_1,g_1],[p_2,g_2])$.  This shows that $\mu$ is surjective.  

To show that $\mu^{-1}$ is continuous, we need some ingredients first.  Let $\delta\colon P_b\times_M P_b\to\gbb$ be the continuous map $\pr_2\circ\mu_b^{-1}$.  For fixed $(p_1,p_2)\in P_b\times_M P_b$, $\delta(p_1,p_2)$ is the unique element of $\gbb$ satisfying $p_1\delta(p_1,p_2)=p_2$.  From this uniqueness it follows that $\delta$ satisfies the identity
\begin{equation}\labell{e:delta}
\delta(p_1h_1,p_2h_2)=h_1^{-1}\delta(p_1,p_2)h_2
\end{equation}
for all $(p_1,p_2)\in P_b\times_M P_b$ and $h_1,h_2\in\gbb$.  Next, the fibered product $(P_b\times\gub)\times_M(P_b\times\gub)$ with respect to the quotient map $\pi\circ\pr_1\colon P_b\times\gub\to M$ admits a $(\gbb\times\gbb)$-action, given by $$(h_1,h_2)\cdot((p_1,g_1),(p_2,g_2))=((p_1h_1^{-1},h_1g_1),(p_2h_2^{-1},h_2g_2)),$$ whose orbit space is $P'\times_M P'$. Define
$$D\colon (P_b\times\gub)\times_M(P_b\times\gub)\longrightarrow\g^1 $$  by the assignment  $$((p_1,g_1),(p_2,g_2))\longrightarrow g_1^{-1}\delta(p_1,p_2)g_2.$$  Then $D$ is continuous, and it follows from Equation~\eqref{e:delta} that it is $(\gbb\times\gbb)$-invariant.  Thus $D$ descends to a continuous map $\widehat{D}\colon P'\times_M P'\to\g^1$, sending $([p_1,g_1],[p_2,g_2])$ to $g_1^{-1}\delta(p_1,p_2)g_2$.  A straightforward calculation shows that $$[p_1,g_1]\widehat{D}([p_1,g_1],[p_2,g_2])=[p_2,g_2],$$
and so we have $\mu^{-1}([p_1,g_1],[p_2,g_2])=([p_1,g_1],\widehat{D}([p_1,g_1],[p_2,g_2])$.  In particular, $\mu^{-1}$ is continuous.

To prove the last claim, let $\pi\colon P\to M$ be a principal $\g$-bundle, let $\pi_b\colon P_b\to M$ be the restricted principal $\gbb$-bundle, and let $\pi'\colon P'\to M$ be the principal $\g$-bundle constructed as above from $\pi_b$.  By Proposition~\ref{p:formY} and Remark~\ref{i:formP} it is straightforward to check that $\pi\colon P\to M$ is isomorphic as a principal $\g$-bundle to $\pi'\colon P'\to M$.  On the other hand, given a principal $\gbb$-bundle $\pi_b\colon P_b\to M$, let $\pi'\colon P'\to M$ be the principal $\g$-bundle constructed as above, and let $\pi'_b\colon P'_b\to M$ be the restricted principal $\gbb$-bundle.  Again, it is straightfoward to check that the map $\Psi\colon P_b\to P'_b$ sending $p$ to $[p,i(b)]$ is a $\gbb$-homeomorphism such that $\pi_b(p)=\pi'_b(\Psi(p))$ for all $p\in P_b$.
\end{proof} 

\begin{example}\labell{x:universal bdle}
Let $\g$ be a transitive groupoid, fix $b\in\g^0$, and suppose that $t|_{\glb}\colon\glb\to\g^0$ admits local sections.  Let $\rho\colon E\gbb\to B\gbb$ be the universal $\gbb$-bundle.  Then we obtain a principal $\G$-bundle $\rho'\colon E\g\to B\gbb$, where $E\g:=E\gbb\times_\gbb\gub$, which we call the \textbf{universal $\g$-bundle}.  This is justified by the following proposition, although we need a definition first.
\end{example}

\begin{definition}\labell{d:pullback bdle}
Let $\pi\colon P\to M$ be a principal $\g$-bundle with anchor map $\epsilon$, and let $F\colon N\to M$ be a map.  The \textbf{pullback bundle} $F^*P\to N$ is the principal $\g$-bundle with total space $$F^*P=\{(x,p)\in N\times P\mid F(x)=\pi(p)\},$$ anchor map $\epsilon\circ\pr_2$, and action $(x,p)g=(x,pg)$.
\end{definition}

\begin{proposition}\labell{p:EG}
Let $\G$ be a transitive groupoid with compact stabilizers, fix $b\in\G^0$, and suppose that $t|_{\glb}\colon\glb\to\g^0$ admits local sections.  Any principal $\g$-bundle $P\to M$ is isomorphic to a pullback of $\rho'\colon E\G \to B\gbb$ to $M$, and isomorphism classes of principal $\g$-bundles over $M$ are in bijection with homotopy classes of maps $M\to B\gbb$.  Moreover, if $\calB_{\g}(M)$ is the set of all isomorphism classes of principal $\g$-bundles over $M$ and $[M,B\gbb]$ the set of all homotopy classes of maps $M\to B\gbb$, then the functors $\calB_{\g}$ and $[\cdot,B\gbb]$ from spaces to sets are naturally isomorphic.
\end{proposition}

\begin{proof}
It follows from Lemma~\ref{i:rest} and Proposition~\ref{p:restric-princ-bundle} that there is a natural isomorphism from $\calB_{\g}$ to $\calB_{\gbb}$.  Compose this with the standard natural isomorphism from $\calB_{\gbb}$ to $[\cdot,B\gbb]$.  The result follows once we show that for a fixed $\psi\colon M\to B\gbb$ that $\psi^*E\g$ is isomorphic to $(\psi^*E\gbb)\times_\gbb\gub$.  By Remark~\ref{i:formP}, $\psi^*E\g$ is isomorphic to $\psi^*(E\gbb\times_\gbb\gub)$.  This in turn is isomorphic to $\psi^*(E\gbb)\times_\gbb\gub$ as principal $\gbb$-bundles with fiber $\gub$ by a standard bundle theory argument.  These bundles are principal $\g$-bundles (Proposition~\ref{p:restric-princ-bundle}), and it is straightforward to check that this isomorphism is $\g$-equivariant.  This completes the proof.
\end{proof}

\begin{corollary}\labell{c:EG}
Let $\G$ be a transitive groupoid with proper target map and stabilizers that are topological manifolds, fix $b\in\g^0$, and suppose that $t|_{\glb}\colon\glb\to\g^0$ admits local sections. The space $E\G$ is a $\G$-CW-complex.
\end{corollary}

\begin{proof}
This follows from Proposition~\ref{p:EG}, Proposition~\ref{P:cw}, the fact that a proper target map implies that the stabilizers of $\g$ are compact, and the fact that $E\gbb$ is a $\gbb$-CW-complex (see \cite[Theorem 1.9]{Lueck-survey}). In particular, since $\gbb$ is a compact group whose underlying topology is a manifold, by von Neumann's solution to Hilbert's fifth problem \cite{vH}, it admits a Lie structure.  By the equivariant triangulation theorem \cite{illman}, the skeleta of the geometric realization model of $E\gbb$ as described in \cite[Sec.\ 16.5]{may2} are thus finite $\gbb$-CW complexes.  Proposition~\ref{P:cw} now applies.
\end{proof}

\end{document}